\documentclass[10pt,bezier]{article}
\usepackage{amsmath,amssymb,amsfonts, euscript}
\usepackage{amscd}
\usepackage{amsthm}
\usepackage{fancybox}
\usepackage[all]{xy}
\newtheorem{theorem}{Theorem}[section]

\newtheorem{corollary}[theorem]{Corollary}
\newtheorem{example}[theorem]{Example}
\newtheorem{definition}[theorem]{Definition}
\newtheorem{remark}[theorem]{Remark}

\newtheorem{proposition}[theorem]{Proposition}
\newcommand{\gr}{\mbox{gr}}

\newcommand{\m}{\frak{m}}
\newcommand{\n}{\frak{n}}

\def\Syz{{\rm Syz}}

\def\Lex{{\rm Lex}}

\def\pd{{\rm pd}}

\title{\bf\Large Consecutive  cancellations in Betti numbers of local rings }
\author{\large    Maria Evelina  Rossi  and  Leila Sharifan  }
\date{ }
\begin{document}
\maketitle

\begin{center}{\it Department of Mathematics,
University of Genoa,\\ Via Dodecaneso 35, 16146 Genoa, Italy\\
rossim@dima.unige.it}
\end{center}

\begin{center}{\it Faculty of Mathematics and  Computer Science, Amirkabir University of Technology\\
424, Hafez Ave.,  15914 Tehran, Iran   \\
leila-sharifan@aut.ac.ir}
\end{center}

\begin{abstract}
Let $I$ be  a  homogeneous ideal in a polynomial ring $P$ over a
field. By Macaulay's Theorem, there exists a lexicographic ideal $L=
\Lex(I) $ with the same Hilbert function as $I. $  Peeva has proved
that the  Betti numbers of $P/I $  can be obtained from the graded
Betti numbers  of $P/L  $  by a suitable sequence of   consecutive
cancellations. We extend this result to any  ideal $I$ in a regular
local ring $(R,\n) $ by passing through the associated graded ring.
To this purpose it  will be necessary to enlarge the list of the allowed cancellations. Taking  advantage of   Eliahou-Kervaire's  construction, several
applications are presented.  This connection between the graded
perspective and the local one is a new  viewpoint  and we hope it
will be useful for studying the numerical invariants  of classes of local
rings.
\end{abstract}

 \section{Introduction}

For a given ideal  $I$  of a regular local ring $(R, \n), $ it is a
central problem to get information on the homological invariants of
the quotient ring $A=R/I $  under reasonable conditions on the ideal
$I. $ Often,  by means of the Hilbert function of $A, $ we  try to
find information  for  $ A $ being Cohen-Macaulay, Gorenstein  or
having estimated depth. Deeper information can be deduced from a
minimal free resolution of $A$ as an $R$-module.   If $I$ is a
homogeneous ideal in a polynomial ring $ P $ over a field, by
Macaulay's Theorem \cite{M}, there exists a lexicographic ideal $L= \Lex(I) $
with the same Hilbert function as $I.$  A  result by  Bigatti
\cite{Bi}, Hulett \cite{Hu}  and Pardue \cite{P} says that the
graded Betti numbers $\beta_{ij}(P/L)$ are greater than or equal to
the corresponding graded Betti numbers $\beta_{ij}(P/I)$.
Peeva \cite[Theorem 1.1.]{Pe}  proved that  the graded Betti numbers
$\beta_{ij}(P/I)$ can be obtained from the graded Betti numbers
$\beta_{ij}(P/L )$ by a sequence of   {\it zero  consecutive
cancellations}, which are cancellations in the graded Betti numbers
of consecutive homological degrees corresponding to the same shift.

The  first aim of this paper is  to complete Peeva's result   in the local case.
 For any ideal $I$, the Hilbert function of the local ring $A=R/I $ is the
Hilbert function of the associated graded ring   $
gr_{\m}(A):=\oplus_{t\ge 0}\m^t/\m^{t+1}    $ where $\m=\n/I.$ In
particular $ gr_{\m}(A)\simeq P/I^* $ where $I^*$ is the ideal of
the polynomial ring $P=\gr_{\n}(R) $ generated by the initial forms
of the elements of $I.$ Then there exists a  unique lexicographic
ideal $L=\Lex(I) $ such that $P/L $  has the same Hilbert function
as $ A. $ Starting from a graded free resolution of  $ gr_{\m}(A) $
as a $P$-module,  we can build up a free resolution of $A$ as an
$R$-module which is not necessarily minimal.
   It will be enough to enlarge the list of the allowed cancellations on the resolution of $L=Lex(I) $ for getting a resolution of a local ring $R/I. $
   The crucial point will be to  prove that the Betti numbers of $ A $ can be obtained from the graded Betti numbers of   $ gr_{\m}(A) $  by a sequence of  {\it negative consecutive cancelations} (see Theorem \ref{asli}).    Most of the results are presented in the more general setting of the filtrations  on a module $ M$ over a regular local ring $(R, \n).  $ This generality  enables us  to detect the minimal free resolution of the same  module    through the associated graded modules of suitable   good filtrations.       For instance, if    the   ideal $I$ is $2$-generated, then $gr_{\m}(A) $ can be very complicated,  but   it  will be  convenient  to get information about the free resolution of $I$ (as an $R$-module)  by considering the $\n$-adic filtration on $I$ itself as an $R$-module. In this case the associated graded module  $gr_{\n}(I)= \oplus I\n^t/I \n^{t+1} $  has the same Betti numbers as $I $  (see \cite{HRV}).

Taking advantage of  the  Peeva's result,   we   can describe in
Theorem \ref{asli2},  the admissible  consecutive cancellations in
the minimal free resolution of $L=\Lex(I) $ in order to achieve  a
minimal free resolution of the local ring $ A.  $  It will be enough to consider the new cancellations coming from Theorem \ref{asli}.  Here, we present
 several  applications  to find  local rings of homogeneous type  (Corollary \ref{successive degree}), to control  the
 depth of $A$  (see Corollary \ref{Lexi}) and to study  the admissible  Hilbert functions of special classes of  local
 rings (Corollary \ref{gorenstein}). The fact that   the lexicographic ideal $L$ is   stable,  so its minimal  free
 resolution is given by the well known Eliahou-Kervaire's construction, often plays a key
 role in the above results.

 As well as being in the homogeneous context, it should be noted that there are many examples in which the existence of
  possible consecutive cancellations does not imply the existence of an ideal that those cancellations are realized for it.
 This is not the case if we consider a perfect  ideal of codimension two.  Actually, in this situation, we prove that for each sequence of
 zero or negative consecutive  cancellations    on the minimal graded free resolution of $ L, $ we can realize an ideal
  $I $   in the formal series  ring  so  that $L= \Lex(I) $ including such  resolution (see Remark  \ref{due}). Further,
 we will give a short proof to a  well-known  characterization of the admissible Hilbert functions of an  Artinian  Gorenstein  local ring of codimension two.

     The paper presents several examples, all   performed  using CoCoA \cite{C}.

 \vskip 2mm

\section{Preliminaries}

Throughout the paper $(R, \n)$ is a regular local ring with infinite  residue field $k. $ If dim$R=n, $ then the associated graded ring $gr_{\n}(R) $ with respect to the $\n$-adic filtration is the  polynomial ring $P=k[x_1, \dots, x_n].$
 
Let $M$ be a finitely generated $R$-module.  We say, according to
the notation in \cite{RV}, that a filtration of submodules
$\mathbb{M} = \{ M_n\}_{n\ge 0}  $ on $M$ is  an $\n$-filtration if
$\n M_n \subseteq M_{n+1} $ for every $n \ge 0, $ and a good (or
stable) $\n$-filtration if $\n M_n =  M_{n+1} $ for all sufficiently
large $n.$  In the following a {\it{ filtered module}} $M$ will be
always an $R$-module equipped with a good $\n$-filtration
$\mathbb{M}.$

If    $\mathbb{M}=\{M_j\}$ is an  $\n$-filtration of $M$, define
$$gr_{\mathbb{M}}(M)=\bigoplus_{j\ge 0}(M_j/M_{j+1})$$ which is a graded
$gr_{\n}(R)$-module in a natural way. It is called the {\bf
associated graded module} to the filtration $\mathbb{M} $ and, for
short, it could   also be denoted by $M^*.$ To avoid triviality, we
assume that $gr_{\mathbb{M}}(M)$ is not zero or equivalently $M \not
= 0.$

\vskip 2mm

 If $m \in M\setminus\{0\}, $ we denote by $\nu_{\mathbb{M}}(m) $ the largest integer $p$ such that  $ m \in M_p $ (the so-called valuation of $m$ with respect to $\mathbb{M}) $ and we denote by $m^* $ or $gr_{\mathbb{M}}(m) $ the residue class of $m$ in $M_p/M_{p+1} $ where $p= \nu_{\mathbb{M}}(m)$ and call it the initial form of $m$ with respect to ${\mathbb{M}}$.   If $m=0, $ we set   $\nu_{\mathbb{M}}(m)= + \infty. $

\vskip 2mm
 If $N$ is a submodule of $M,$ by Artin-Rees Lemma, the sequence $\{N\cap M_j\ | \ j\ge 0\}$ is a good $\n$-filtration
 of $N$.  Since  \begin{equation}\label{N}  (N\cap M_j)/(N\cap M_{j+1})\simeq (N\cap M_j+M_{j+1})/M_{j+1}, \end{equation}   $gr_{\mathbb{M}}(N)$ is a graded submodule of $gr_{\mathbb{M}}(M). $

   Using (\ref{N}), it is clear that  $gr_{\mathbb{M}}(N)$ is generated by the elements $x^* $ with $x \in N.$ We write $$ gr_{\mathbb{M}}(N) =< x^* \   : \ x \in N>.$$
   On the other hand,  it is clear that $\{(N+M_j)/N \ | \ j\ge 0\}$ is a good $\n$-filtration of $M/N$ which we denote  by $\mathbb{M}/N.$ These graded modules are related by the graded isomorphism $$gr_{\mathbb{M}/N}(M/N)\simeq gr_{\mathbb{M}}(M)/gr_{\mathbb{M}}(N).$$
   \vskip 3mm

For a given filtered module $M$, we recall that an element $g \in M$
is a {\it{lifting}}  of an element $h \in gr_{\mathbb{M}}(M) $ if
$gr_{\mathbb{M}}(g)=h.$

The morphism of filtered modules $f : M \to N $ ( $f(M_p) \subseteq
N_p $ for every $p$) clearly induces a morphism of graded
$gr_{\n}(R)$-modules $$gr(f) :  gr_{\mathbb{M}}(M) \to
gr_{\mathbb{N}}(N).$$ Then  $ gr_{\mathbb{M}}(.) $ is a functor from
the category of filtered $R$-modules into the category of the graded
$gr_{\n}(R)$-modules.  Furthermore, we have a canonical embedding $
gr_{\mathbb{M}}(Ker f)  \to Ker( gr(f)).$

 Let $F= \oplus_{i=1}^s R e_i $ be a free $R$-module of rank $s $ and $\nu_1, \dots, \nu_s $ be integers. We define the filtration $\mathbb{F} = \{F_p : p \in {\bf{Z}} \} $ on $F$ as follows
 $$F_p := \oplus_{i=1}^s \n^{p- \nu_i} e_i = \{(a_1, \dots , a_s) : a_i \in \n^{p- \nu_i} \}.$$
 From now on, we denote the filtered free $R$-module $ F$ by $ \oplus_{i=1}^s R(-\nu_i) $ and we call it  a {\it{ special filtration}} on $F.$ If  $({\bf{F.}}, \delta .) $ is a complex of finitely generated free $R$-modules, a special filtration on  {\bf{F.}}  is a special filtration on each $F_i $ that makes $({\bf{F.}}, \delta .) $ a  complex of filtered modules.

  Let $M$ be a   finitely generated  filtered $R$-module  and let  $S=\{f_1, \dots, f_s\} $ be a system of elements  of $M $ and   $ \nu_{\mathbb{M}}(f_i) $  be the corresponding valuations.  As before  let $F =  \oplus_{i=1}^s R e_i $ be a free $R$-module of rank $s$ equipped with the filtration  $\mathbb{F}    $ where  $\nu_i=  \nu_{\mathbb{M}}(f_i).$   Then  we denote the filtered free $R$-module $F $ by $ \oplus_{i=1}^s R(-  \nu_{\mathbb{M}}(f_i) ), $ hence $\nu_{\mathbb{F} }(e_i)=  \nu_{\mathbb{M}}(f_i).$

 Let $\phi  : F  \to M $ be a morphism of filtered $R$-modules   defined by $$\phi  (e_i)= f_i. $$
 Denote by $Syz (S) $ the submodule of $F $ generated by the first syzygies of $f_1, \dots, f_s, $ then
 $\Syz(S) =$Ker$ \phi. $

\begin{definition}
Let $M$ be a filtered module. A subset $S=\{f_1,..,f_s\}$ of $M$ is called a standard basis of $M$ if $gr_{\mathbb{M}}(M)=\langle f_1^*,...,f_s^*\rangle.$  \end{definition}

 By following the initial idea of Robbiano and Valla in \cite{RoV}, Shibuta in \cite{Sh}  characterized
 the  standard bases of a filtered module    as   follows:

\begin{theorem}\label{esiste}Let  $M$ be a filtered $R$-module, $f_1,\dots,f_s \in M$ and $S=\{ f_1,\dots,f_s\} $.
The following facts are equivalent:
\begin{enumerate}\item  $\{f_1,\dots,f_s\}$ is  a   standard basis of $M.$
\item $\{f_1,\dots,f_s\}$ generates $M $ and every element of $ \Syz(\langle f_1^*,...,f_s^*\rangle) $ can be lifted to an element  in $Syz(S).$
\item $\{f_1,\dots,f_s\}$ generates $M $ and $  \Syz(\langle f_1^*,...,f_s^*\rangle) = gr_{\mathbb{F}}(\Syz(S)).$
 \end{enumerate}
\end{theorem}
 \vskip 2mm
The valuation of an element  in a filtered $R$-module
often  plays   the same  role  as  the degree of a
homogeneous element in a  $P$-module, but we want to point out that there are
important  differences. For example, in a  graded $P$-module  the number   of the
   homogeneous generators in each degree  is  independent from the chosen  minimal
  system;   the  same property does not hold in
the local case with the valuations.  For instance, in $A=k[[x,y,z]], $
it is easy to verify that  $\{xy + z^3, x^2 + xy^3, xz^3 -xy^4 +
y^2z^4\} $ and $ \{xy + z^3, x^2 + xy^3, y^2z^4\} $  are two minimal
system of generators of  the same ideal $I,  $ but they have
different valuations with respect to the classical $\m$-adic
filtration,  even if  they are liftings   of
minimal generators of $I^* =(xy,   x^2, xz^3 ,  y^2z^4,  -z^6,
y^6z^3),  $ the ideal generated by the initial forms of the elements in $I.$

 Using Theorem \ref{esiste},  one can prove   (see \cite{RSh}, Theorem 1.8.) the following result proved by Robbiano in \cite{R}.

\begin{theorem} \label{main}  Let $M$ be a filtered $R$-module and let $({\bf{G.}}, d. )$ be a  $P$-free graded   resolution of
$ gr_{\mathbb{M}}(M) :
$
$$ {\bf{G.}}: \  0 \to \oplus_{i=1}^{\beta_{l}} P(-a_{l i})  \overset{d_{l}} \to   \oplus_{i=1}^{\beta_{l-1}} P(-a_{l-1 i}) \overset{d_{l-1 }}  \to \dots  \overset{d_1} \to   \oplus_{i=1}^{\beta_{0}} P(-a_{0 i})  \overset{d_{0 }} \to gr_{\mathbb{M}}(M) \to  0.$$
Then we can build up an  $R$-free resolution  $({\bf{F.}, \delta.})$ of $M$ and a special filtration $\mathbb{F} $ on it such that $gr_{\mathbb{F}}({\bf{F.}}) = {\bf{G.}}.$
\end{theorem}

We recall that $({\bf{F.}, \delta.})$ is defined by an inductive
process. Let us present  the inductive steps   because the
construction  will be useful in the following  (for more detail we
refer to \cite{RSh}).  Starting from  $({\bf{G.}},d. ), $ denote by
$\{\epsilon_{0i}\}$ a basis of $G_0.$ We put $g_i=d_0(\epsilon_{0i})
\in gr_{\mathbb{M}}(M)   $ and let $f_i \in M $ be such that
$gr_{\mathbb{M}}(f_i)=g_i. $ Then   $ a_{0i} =
\nu_{\mathbb{M}}(f_i). $ We define   the $R$-free module $F_0$ of
rank $\beta_0 $ with the induced special filtration ${\mathbb{F}_0}
$  on $R^{\beta_0}$  $$F_0= \oplus_{i=1}^{\beta_0} R(-a_{0 i})  $$
Denote by $\{e_{0i}\}$ a basis of $F_0$ and define  $ \delta_0: F_0
\to M $   such that $\delta_0(e_{0i})=f_i. $ Since $d_0 $ is
surjective,  the $f_i$'s generate a standard basis of $M,   $  $
Ker(d_0) = gr_{\mathbb{F}_0}(Ker (\delta_0))  $ and $ F_0
\overset{\delta_0}\to M \to 0 $ is exact.   We can repeat the same
procedure on the successive $j$-steps ($j >0$) of the resolution of
$gr_{\mathbb{M}}(M).    $  We lift  a system of generators of $Ker
(d_{j-1}) = gr_{\mathbb{F}_{j-1}}(Ker ( \delta_{j-1}) ) $ to
elements in $Ker (\delta_{j-1}) $ of valuation $a_{j i}. $ Hence we
build up the $R-$free modules  $F_{j}$  with special filtrations
$F_{j}=\oplus_{i=1}^{\beta _{j }}R(-a_{j i}) $ and the differential
maps $\delta_j : F_j \to F_{j-1} $ such that
 $gr_{\mathbb{F}_j}(F_j)  = G_j , $ $gr_{\mathbb{F}_j}(\delta_j)  = d_j . $ Since, by construction, the lifted elements    form a standard basis of  $Ker (\delta_{j-1}),  $ we have
 \begin{equation} \label{step} Ker (d_j)=gr_{\mathbb{F}_j}(Ker ( \delta_j) ) \end{equation} and we go on.

\vskip 2mm

It is worth   saying  that  the $R$-free resolution of $M$
$$   {\bf{F.}}: \  0 \to R^{\beta_{l}}   \overset{\delta_{l}} \to  R^{\beta_{l-1}}   \overset{\delta_{l-1 }}  \to \dots  \overset{\delta_1} \to  R^{\beta_{0}}   \overset{\delta_{0 }} \to M \to  0,  $$ coming   from a minimal free resolution of $ gr_{\mathbb{M}}(M)$,    is not necessarily minimal.  In particular $({\bf{F.}, \delta.})$ is minimal if and only if  the Betti numbers of $M$ and $gr_{\mathbb{M}}(M) $ coincide. Our aim is to get information from the numerical invariants of ${\bf{G.}} $ in order to achieve  a minimal $R$-free resolution of $M.$

\vskip 3mm

Let $N$ be a finitely generated graded module over the polynomial
ring $P$. We consider   $$  G_j =  \oplus_{i=1}^{\beta_j} P(-a_{ji})
\overset{d_{j }}  \to G_{j-1} =  \oplus_{i=1}^{\beta_{j-1}}
P(-a_{j-1, i}),$$ a part of      a  {\it minimal } free resolution
$({\bf{G.}}, d.)$  of $N $  with $a_{j 1} \le \dots \le a_{j
\beta_j}  \ $ and  $a_{j-1, 1} \le \dots \le a_{j-1,  \beta_{j-1}}$.
Let   $ 1 \le s \le \beta_j $ and $ 1\le r \le \beta_{j-1} $ and set
$$ u_{r s}: = a_{j s} - a_{j-1, r}, $$ then the matrix $U_j= (u_{r
s}) $ is  called the {\it{j-th degree-matrix }} of $N $. We say that
$U_j $ is {\it non-negative}  if all the entries of $U_j$ are
non-negative.  \noindent We remark that the matrices $U_j$ are
univocally determined by  the graded $P$-module $N. $  Denote by
$\pd (N) $  the projective  dimension of   $N $ as a $P$-module.
\vskip 2mm
\begin{proposition}\label{degree} With the above notations,
let $({\bf{F.}, \delta.})$ be a free resolution of $M$ coming  from
a  graded minimal free resolution $({\bf{G.}}, d.)$  of
$gr_{\mathbb{M}}(M).  $  If   the degree-matrices $U_j$ of
$gr_{\mathbb{M}}(M) $ are non-negative  for every $\  j \le \pd
(gr_{\mathbb{M}}(M)),  $   then
 $({\bf{F.}, \delta.})$ is  minimal.   \end{proposition}
\begin{proof}
Let  $1\leq j\leq \pd(gr_{\mathbb{M}}(M)),  $    $1 \le s \le
\beta_j $ and $ 1\le r \le \beta_{j-1} $ and denote  by
${\mathcal{M}}_j^*=(n_{r s})  $  the matrix associated to $d_j$ with
respect to  the bases $\{\epsilon_{j 1}, \dots, \epsilon_{j \beta_j} \} \ $ of
$G_j$ and  $\{\epsilon_{j-1, 1}, \dots, \epsilon_{j-1,  \beta_{j-1}}\}   \ $ of $
G_{j-1}.$  Let $U_j= (u_{r s}) $ be the corresponding degree-matrix.
Notice that  $deg (n_{r s})= u_{r s} $ if $u_{r s} > 0, $ $n_{r s}=0
$ otherwise.

Following Theorem \ref{main}, we   build up a  free resolution
$({\bf{F.}, \delta.})$ of $M$ and denote by $
{{\mathcal{M}}_j}=({m_{r s}}) $ the corresponding matrix associated
to $\delta_j.  $   By construction (see (\ref{step})) the columns of
${\mathcal{M}}_j$ form a  standard bases of $\Syz_j(M)$ and the
columns of ${\mathcal{M}}_j^*  $ are the initial forms of the
corresponding columns of ${{\mathcal{M}}_j} $ with respect to the
filtration which has been defined on $F_{j-1}.  $   In particular,
the degree matrix $U_j$ controls the valuations of entries of  $
{{\mathcal{M}}_j}$ with respect to the $\n$-adic filtration. If
$n_{r s} \neq  0, $ then $\nu_{\n}({m_{r s}}) = u_{r s}, $ otherwise
$\nu_{\n}({m_{r s}}) >  u_{r s}.  $ So if for all $j$ the entries of
the matrix $U_j$ are non-negative, then all the entries of
${\mathcal{M}}_j$ belong to $\n$ and hence $({\bf{F.}, \delta.})$ is
minimal.
\end{proof}

 The converse of Proposition \ref{degree} is not true.   It is possible that the degree-matrices of $gr_{\mathbb{M}}(M)   $    have  negative entries, nevertheless the built up resolution $({\bf{F.}, \delta.})$ is minimal. This is the case  if we consider  the local ring  $A= k[[t^9, t^{17}, t^{19}, t^{39}]] $ equipped with the $\m$-adic filtration where $\m$ is the maximal ideal of $A$   (see \cite{RSh} Example 3.5.).\vspace{.2cm}

\bigskip

\section{Consecutive cancellations in   Betti numbers of $gr_{ \mathbb{M}}(M) $}

The aim of this section is to describe the possible minimal $R$-free
resolutions  of $M$ obtained  from a minimal $P$-free resolution of
$gr_{\mathbb{M}}(M). $
We present the following definition which is a suitable adaptation of Peeva's definition. \vskip 2mm Given a sequence of
numbers $\{ c_i\} $ such that  $c_i=\sum_{j\in {\bf{N}}} c_{ij}, $
we obtain a new sequence by a {\it consecutive cancellation}  as
follows: fix an index $i, $ and choose $j $ and $j'$  such that $ j
\le  j' $ and $ c_{ij}, c_{i-1, j'} >0;  $ then  replace $ c_{ij}  $
by $ c_{ij}-1   $ and $ c_{i-1,j'}   $ by  $ c_{i-1,j'}-1,    $ and
accordingly, replace in the sequence  $c_i $ by $c_i -1$ and
$c_{i-1}$ by  $ c_{i-1} -1. $
  If $ j <j' $  we call it an {\it i negative consecutive
cancellation}  and if $ j = j'$  an  {\it i zero consecutive cancellation}.
 \vskip 2mm
A {\it sequence of consecutive cancellations } will mean  a finite number of  consecutive cancellations performed on  a given sequence.
\vskip 3mm

 Let $N $ be a  homogeneous  $P$-module   with
  $P$- free graded   resolution given by:
 $$ {\bf{G.}}: \  0 \to \oplus_{i \ge 0 }  P^{\beta_{l j}} (-j)  \overset{d_{l}} \to  \oplus_{i \ge 0 }  P^{\beta_{l-1, j}} (-j)  \overset{d_{l-1 }}  \to \dots  \overset{d_1} \to   \oplus_{i \ge 0 }  P^{\beta_{0 j}} (-j).    $$   According to the above definition,  we will say that   the sequence of the Betti numbers $\{\beta_i = \sum_{j\in {\bf{N}}} \beta_{ij}\}     $ of $ N $
 admits an  {\it  $i $ negative consecutive cancellation }  (resp.   {\it $i $ zero  consecutive cancellation })  if there exist integers   $  j < j'  $ (resp. $j=j' $) such that  $\beta_{i j}, \beta_{i-1,j'}>0.$
 \vskip 2mm
 For example   $ \dots \to P (-3) \oplus P(-5)\oplus P(-6)  \to P^2(-2) \oplus P^2(-5)  \to \dots $ admits a zero cancellation ($P(-5), P(-5)$) and a negative cancellation $(P(-3), P(-5))$.
 \vskip 2mm
 Notice that  an  $i$   negative consecutive cancellation  corresponds  to a  negative entry of the $i$-th degree-matrix   of the graded resolution of $N.$  It is clear that   a sequence of Betti numbers can admit    different diagrams of  zero or negative successive cancellations.
 
  \vskip 2mm
  Combining Theorem \ref{main} and Proposition  \ref{degree}, we  prove  the following result.

\begin{theorem}\label{asli}
Let $(R,\n) $ be a regular local ring and let $M$ be   a filtered
$R$-module. Then  the Betti numbers of $M$ as an $R$-module can be
obtained from the  Betti numbers of $gr_{ \mathbb{M}}(M) $ as a
$gr_{\n}(R)$-module  by a sequence of {\bf negative}  consecutive
cancellations. \end{theorem}

\begin{proof}
  Let $({\bf{G.}}, d.)$ be the  minimal free resolution  of $gr_{ \mathbb{M}}(M)$ and  $\{\beta_i = \sum_{j\in {\bf{N}}} \beta_{ij}\}     $  the corresponding sequence of the Betti numbers.  By Theorem \ref{main},  we build up a
free resolution $({\bf{F.}, \delta.})$ of  $M$ as an $R$-module
from $({\bf{G.}}, d.)$.     If $({\bf{F.}, \delta.})$ is not
minimal, then, for some integer $i, $    the matrix   of the $i$-th
differential map $\delta_i$
 $$F_i=R^{\beta_i} \overset{\delta_{i }}  \to F_{i-1}=R^{\beta_{i-1}}$$
  has an invertible entry.    Fix  the bases
  $\{e_{i 1}, \dots, e_{i \beta_i} \} \ $ of $F_i$ and  $\{e_{i-1, 1}, \dots, e_{i-1,  \beta_{i-1}}\}   \ $ of
   $ F_{i-1}$ and let ${{\mathcal{M}}_i}=({m_{r s}})$ be the corresponding matrix.

By Proposition
\ref{degree},
  the
$i$-th degree-matrix $U_i=(u_{rs})  $   of $gr_{
\mathbb{M}}(M) $  has a negative entry. This means that the sequence $\{\beta_i = \sum_{j\in {\bf{N}}} \beta_{ij}\}     $
  admits an  $i $ negative consecutive cancellation.
 Let $ (p,q)$ be the least couple such that $m_{pq}$ is invertible
in $ {{\mathcal{M}}_i}. $
 Now, to   $({\bf{F.}, \delta.})$, we  apply   a standard procedure presented for example  in \cite{E2}.  We replace the basis of $F_{i-1}$ by $e'_{i-1, j}= e_{i-1, j}$ if  $1\leq j\leq \beta_{i-1}, $   $j\neq
 p$, and
  $e'_{i-1, p}=\sum _{j=1}^{\beta_{i-1}}m_{jq}e_{i-1,j}. $

 It is clear that,  with respect to the new basis,  the matrices of differential maps in the resolution of $M$  change just for $\delta_i $ and  $\delta_{i-1}$.
 Precisely, one can easily check that in ${{\mathcal{M}}_i}$ for each $r\neq p$, $m_{rs}$ is replaced  by $m_{rs}-m_{pq}^{-1}m_{ps}m_{rq}$ and the entries of the the $p-$th
 row, $m_{ps}$, are replaced  by $m_{pq}^{-1}m_{ps}$. So \\
1) the $q-$th column of ${{\mathcal{M}}_i}$ is replaced by   $(\overbrace{0,...,0}^{p-1},1,0,...,0)^{tr}$.\\
 2) the $p-$th column of  ${{\mathcal{M}}_{i-1}}$ is replaced  by
$(0\ ... \ 0)^{tr}.$\vspace{.2cm}

Therefore, $\delta_j(e_{iq})=e'_{i-1,p}$ and
$\delta_{i-1}(e'_{i-1,p})=0$. Let $H_j=0$ if $j\neq i-1,i$ and
$H_i=F_i|_{e_{iq}}$ and $H_{i-1}=F_{i-1}|_{e'_{i-1,p}}$. Thus we
have found the following trivial subcomplex of $({\bf{F.},
\delta.})$
$${\bf {H.}}:\hspace{.1cm}\underbrace{0\rightarrow \cdots \rightarrow
0}_{l-i+1}\to R\overset{id} \to R \rightarrow\underbrace {0
\rightarrow\cdots \rightarrow 0}_{i},$$ where $l$ is the length of
$({\bf{F.}, \delta.}). $  ${\bf {H.}} $ is embedded in ${\bf {F.}}$
in such a way that $\widetilde{{\bf {F.}}}= {\bf {F.}}/{\bf {H.}}$
is again a free resolution of $M$ which corresponds to cancelling a
copy of $R$ in $F_i $ and $F_{i-1}.$  The matrices of differential
maps of  ${\bf {F.}}/{\bf {H.}}$ are  different from those of
$({\bf{F.}, \delta.})$ just for $i-1, i, i+1$. Denote the matrices
of the new resolution with respect to the new bases by
$\widetilde{{\mathcal{M}}}_j. $   Then
$\widetilde{{\mathcal{M}}}_{i-1}$ is obtained
 by deleting  the $p-$th column of ${{\mathcal{M}}}_{i-1}$,
$\widetilde{{\mathcal{M}}}_{i+1}$ is obtained
 by deleting  the $q-$th row of ${{\mathcal{M}}}_{i+1}$, and,  finally, by  deleting the
$q-$th column and $p-$th row of ${{\mathcal{M}}}_{i}$ we obtain
$\widetilde{{\mathcal{M}}}_{i}$.  It is easy to see that the
eventually remaining   invertible entries of the matrices of the
differential maps of the new resolution,
$\widetilde{{\mathcal{M}}}_j$, still correspond to the negative
entries of the degree matrices of $gr_{ \mathbb{M}}(M)  $ out of the
$p-$th row and $q-$th column of ${\mathcal{M}}_i^*.   $  We can
repeat the procedure on $\widetilde{{\bf {F.}}} $ until  getting  a
minimal free resolution of $M. $

\end{proof}

In general, the Betti numbers of $gr_{\m}(A) $ can be much greater
than those of $A. $ For example, if we consider $A=R/I $ with $I=
(x^3-y^7, x^2y-xt^3-z^6) \subseteq R=k[x,y,z,t], $ then
$\beta_1(A)=\mu(I)=2 $ and $\beta_1(gr_{\m}(A))=\mu(I^*)=8.$
However, Theorem \ref{asli} gives a constructive method which
relates the minimal free resolution of $A$ and   the minimal  graded
free resolution of $gr_{\m}(A).$

We present the following example in order to help the reader to visualize better the
procedure of Theorem \ref{asli}.
\begin{example} {\rm   Consider     $A=k[[t^{19}, t^{26}, t^{34}, t^{40}]] \simeq R /I $   where  $R= k[[x, y, z, w]]. $ Using Theorem \ref{asli},  we want to deduce the Betti numbers of $A$ as an $R$-module from the   minimal graded free resolution of $gr_{\m}(A) = P/I^* $ as a $P= k[x, y, z, w]$-module.

 In this case   we have $\mu(I)=\mu(I^*) =5.  $    Both  $A$ and $ gr_{\m}(A)=P/I^*  $ are Cohen-Macaulay, hence they have the same homological dimension. Nevertheless they have different Betti numbers:
$$0 \to R \to R^5 \to R^5  \to I   \to 0$$
  and
  $$0 \to  {   P^2 }  \to  {  P^{6}}  \to  P^5  \to  I^* \to 0.$$
We consider the graded minimal free resolution of $I^*:$ $$\  0\to {
P({  -5})} \oplus P(-8)  \to P^3({  -4}) \oplus P^2(-5) \oplus { P({
-6})} \to P^5(-3)  \to I^* .$$ It presents a unique negative
cancellation: $(P({  -5}), P(-6) ). $  Notice that the Betti numbers
of $A$ are obtained from  the total Betti numbers of $\gr_{\m}(A),
$ after performing the above cancellation.  }
\end{example}

 \vskip 3mm

 It  should be noted that there are many examples where the existence of possible consecutive cancellations does not imply the existence of an ideal for which those cancellations are realized.

\begin{example} {\rm  Let $ I $ be an ideal in the regular local ring $(R, \n)$ such that $A=R/I $ is Artinian with Hilbert function $ \{(1, 5, 1, 1, 1)\} $  and $R/\n$ has characteristic $0.$  Elias and Valla (see \cite{EV}) have  proved that the number of the isomorphism classes of the Artinian local rings with this  Hilbert function is $5. $ They  have different Betti numbers because they correspond to the different values  of the Cohen-Macaulay type $1 \le \tau \le 5.$
Up to isomorphism, all of them have the same  associated graded ring
$gr_{\m}(A) =P/I^* $ where
$$ I^* = (x_1^5, x_1 x_2, x_1x_3, x_1x_4, x_1x_5, x_2^2, x_2x_3, x_2x_4, x_2x_5,
x_3^2, x_3x_4, x_3x_5, x_4^2, x_4x_5, x_5^2) $$ in $P=k[x_1, \dots, x_5].$
Hence the minimal free resolution of $gr_{\m}(A) $  is
\begin{eqnarray*}
{\bf{G.}}: 0 \to P^4(-6)\oplus P(-9) \to  P^{20}(-5)\oplus P^4(-8)
\to P^{39}(-4)\oplus P^6(-7)\\  \to P^{36}(-3)\oplus P^4(-6)\to
P^{14}(-2)\oplus  P(-5) \to P.
\end{eqnarray*}
   By Theorem \ref{asli} and by Elias and Valla's result, we know  that only $ 5$ diagrams of negative consecutive cancellations can be   realized, but  the resolution of $gr_{\m}(A) $ admits a larger number of sequences  of negative consecutive cancellations.   }

\end{example}

\vskip 3mm

Next example shows that we may take advantage of the generality of Theorem \ref{asli} by  using  a more   advantageous  filtration.

\begin{example} {\rm  Consider  the ideal $I = (x^2y^5, xyz^6-z^9, y^5z^6) $  in   $  R = K[[x, y, z]].$   We can prove that     $ gr_{\m}(A)= P/I^*  $ where
$ I^*= (x^2y^5,  xyz^6,  y^5z^6,  y^4z^9,  y^3z^{12},  y^2z^{15},
yz^{18},  z^{21}) $   is the ideal of $P = K[x,  y,  z] $ generated
by the initial forms of the elements of $I.$ The  minimal graded
free resolution  of  $I^*$ as a $P$-module is
 \begin{eqnarray*}
{\bf{G.}}: 0 \to P(-15)\oplus P(-17)\oplus P(-19)\oplus P(-21) \to
P(-12)\oplus P(-13)\oplus P^2(-14)  \\ \oplus P^2(-16)\oplus
P^2(-18)\oplus P^2(-20)\oplus P(-22)   \to P(-7)\oplus P(-8)\oplus
P(-11)\oplus P(-13)\\ \oplus  P(-15)\oplus P(-17)\oplus P(-19)\oplus
P(-21)
 \end{eqnarray*}
 which admits several negative cancellations.
\vskip 3mm \noindent By   Theorem \ref{asli},   we can also compute
a resolution of $I$ as an $R$-module, by considering the $\n$-adic
filtration on $I$ itself. Then we are interested  in a resolution of
the graded $P$-module $gr_{\n} (I)= \oplus_n  I \n^n/I \n^{n+1}$
where $ \n=(x,y,z). $ In \cite{HRV}, page 595,   it is   proved that
$$ gr_{\n} (I) = P\oplus P/(x^2y^5) \oplus P/(x^2, xy,xz^3,z^6). $$
  Hence the minimal free resolution of $ gr_{\n} (I)  $ is
\begin{eqnarray*}
{\bf{G.}}: 0 \to P(-6) \overset{d_3}\to P(-3)\oplus P^2(-5)\oplus
P(-7) \overset{d_2} \to P^2(-2)\oplus P(-4)\oplus P(-6)\oplus P(-7)
\overset{d_1} \to P^3
\end{eqnarray*}
which is easier to handle than those of  $gr_{\m}(A)= P/I^*.   $

 After  the negative  consecutive  cancellations  on $(F_3, F_2,
F_1) $ corresponding  to
    $(P(-6),P(-7),0),$ $  (0, P(-5),P(-7)), $ $ (0, P(-5),P(-6) ), $   $ (0, P(-3),P(-4)) $ we get the minimal free resolution of $I $ as   $$ 0 \to R^2 \to R^3 \to I \to 0.$$}

 \end{example}
\medskip

\section{Consecutive cancellations in   Betti numbers of $\Lex(I)  $}

\vskip 3mm Let $I$ be an ideal of the  regular local ring $(R,\n) $
and consider the local ring $ A=R/I $ with maximal ideal $\m=\n/I. $
The Hilbert function of $A$ is the Hilbert function of $
gr_{\m}(A)=P/I^*  $ where $I^* $ is the ideal of the polynomial ring
$P$ generated by the initial forms of the elements of $I.$ We denote
by $L=\Lex(I) $ the (unique)  lexicographic ideal of $P$ such that
$P/L $ has  the same Hilbert function as $ gr_{\m}(A).$ While the
resolution of $ gr_{\m}(A) $ is in general unknown,
  the resolution of  $P/\Lex(I) $ can be easily determined.

 Combining Peeva's result  and Theorem \ref{asli}, we immediately get the
following theorem.

\begin{theorem}\label{asli2} Let $I$ be an ideal of the  regular local ring $(R,\n). $  The Betti numbers of $R/I$ can be obtained from the Betti numbers of $P/\Lex(I) $  by a sequence of  negative and  zero consecutive cancellations.
\end{theorem}

   \vskip 2mm
 In order to apply Theorem \ref{asli2}, we are interested in finding  the   zero and negative  consecutive cancellations  in $\{\beta_{i}(P/L)\}. $ For a monomial $m$, we set $\max(m)=\max\{p| x_p \ {\text divides }\  m\}.$ We can easily get  an extension  of Proposition 1.2  in  \cite{Pe} for testing also the existence of negative cancellations.

\begin{proposition}\label{Eliahou}
Let $L$ be a lexicographic  ideal of $P$. If an $i$ zero or negative consecutive cancellation in the sequence of $\{ \beta_i = \sum_j \beta_{ij}(P/L)\}$ is possible, then the following two conditions are satisfied:
\begin{enumerate}
\item $L$ has a minimal monomial generator $m$ with $\max(m)\geq i-1$.
\item $L$ has a minimal monomial generator $m'$ with $\max(m')\geq i$ and $\deg(m')<\deg(m).$
\end{enumerate}
\end{proposition}

\begin{proof}
Let $G(L)$ be the set of (unique) minimal monomial generators of $L$. The minimal  graded free resolution of $P/L$ is provided by a construction of Eliahou and Kervaire \cite{EK}. The resolution has basis
\begin{equation} \label{ElKe}\{e_{(m; j_1, . . . , j_{i-1})} | m\in G(L), 1\leq j_1 < . . . < j_{i-1} < \max(m)\} \end{equation}
in homological degree $i$ and $e_{(m; j_1, . . . , j_{i-1})}$ has
degree $i-1+\deg(m)$. Thus, the first condition is equivalent to
$\beta_{i-1, i-2+\deg(m)}\neq 0$ for some monomial $m$ in $G(L)$,
and the second condition is equivalent to $\beta_{i,
i-1+\deg(m')}\neq 0$ for some monomial $m'$ in $G(L)$ such that
$\deg(m')<\deg(m).$ Both conditions together are equivalent to the
fact that the degree matrix $U_i$ of $P/L$ has a negative or zero
entry.
\end{proof}

The next corollaries are easy applications of Theorem \ref{asli2}.

\begin{corollary}\label{successive degree}
Let $I$ be an ideal in the regular local ring $(R,\n). $   If
$L=\Lex(I) $ is minimally generated   in two successive degrees,  then
$\beta_i(R/I)=\beta_i(P/I^*)$ for each $i\geq 0$.
\end{corollary}

\begin{proof}
Assume  that $L$ is generated in degrees $a$ and $a+1$. By the proof of Proposition \ref{Eliahou}, the minimal free resolution of $P/L$ has the following shape :
\begin{eqnarray*}
& 0 &\to P^{\beta_{h, a+h-1}}(-a-h+1)\oplus P^{\beta_{h,a+h}}(-a-h)\to ...\to P^{\beta_{2, a+1}}(-a-1)\oplus P^{\beta_{2,a+2}}(-a-2) \\ && \to P^{\beta_{1,a}}(-a)\oplus P^{\beta_{1,a+1}}(-a-1)\to P\to 0.
\end{eqnarray*}

 Therefore the only possible cancellations  in the Betti numbers of $P/L$ are zero consecutive cancellations, hence by Theorem \ref{asli}  the conclusion follows.
\end{proof}

 \vskip 3mm

In the following $ \mu(\ \ ) $ will denote the minimal number of
generators. The next corollary extends to the local case a recent
result by  Hibi and Murai  in \cite{HM}.

\begin{corollary}\label{Lexi}
Let $I\subseteq \n^2 $ be a non-zero  ideal of the  regular local
ring $(R,\n)$ of dimension $n.$  Assume $\mu(\Lex(I))\leq n, $ then

\begin{enumerate}
\item $\dim(R/I) =n-1$.
\item $depth(R/I)=depth(P/I^*)=depth(P/\Lex(I))=n-\mu(\Lex(I))$.
\item $\mu(\Lex(I))=pd(R/I).$
\item  $\beta_h(R/I)=\beta_h(P/I^*)=\beta_h(P/\Lex(I))=1 $ where   $h=pd(R/I). $
\end{enumerate}
\end{corollary}

\begin{proof}
Let $G=\{ m_1,m_2,...,m_h\}$ be the (unique) minimal monomial
generating set for the monomial ideal $\Lex(I)$ in $P= k[x_1,\dots,
x_n] $ where $1< \deg(m_1)\leq \deg(m_2)\leq ...\leq \deg(m_h)$ and
where $m_{i+1}<_{\Lex} m_i$ if $\deg(m_i)=\deg(m_{i+1})$. By
considering $x_1>x_2> \cdots $ Hibi and Murai in Proposition 1.2 of
\cite{HM} have showed that  for every $k \le n $  we have
$m_k=x_1^{s_1}x_2^{s_2}...x_k^{s_k+1}$  where $s_1 = \deg m_1 -
1\geq 1$ and $s_i = \deg m_i - \deg m_{i-1}$ for $i=2,3,...,h$. So
clearly $\Lex(I) \subseteq ( x_1) $ and hence $\dim(P/\Lex(I))=n-1$.
By the Eliahou-Kervaire resolution \cite{EK}, the tail  of the
minimal free resolution of $P/\Lex(I)$ is:
\begin{equation} \label{last betti}0\to P(-h-\deg(m_h)+1)\to P^{h-1}(-h-\deg(m_h)+2)\oplus P(-h-\deg(m_{h-1})+2). \end{equation} The conclusion follows now by  Theorem \ref{asli} and Theorem \ref{asli2}.
\end{proof}
\vskip 3mm

 The  following result demonstrates an application in studying the admissible Hilbert functions of particular classes of local rings.

\begin{corollary}\label{gorenstein}
Let $\{(1,n,h_2,...,h_t,1,...,1,0,0, \dots)\} $ be the Hilbert function of an Artinian Gorenstein local ring $A=R/I$. If $t\geq n$ then $h_t\leq n.$
\end{corollary}
\begin{proof}
Suppose that $h_t>n$. Let $G=\{ m_1,...,m_h\}$ be the (unique)
minimal monomial generating set of  the monomial ideal $\Lex(I)$
where $\deg(m_1)\leq \deg(m_2)\leq ...\leq \deg(m_h)$ and where
$m_{i+1}<_{\Lex} m_i$ if $\deg(m_i)=\deg(m_{i+1})$. One can check
that $m_h=x_n^{s+1}$, where $s$ is the degree of the $h-$vector, and
for $1\leq i\leq n$, $m_{h-i}=x_{n-1}^ix_n^{t-i+1}$. Again by the
Eliahou-Kervaire resolution,  the tail of the minimal free
resolution of $P/\Lex(I)$ is:
 \begin{eqnarray*} 0 &\to &\oplus_{j<n+t}P^{\beta_{n, j}}(-j)\oplus P^{\beta {n, n+t}}(-n-t)\oplus P(-n-s)\to \\ &&
\oplus_{j<n+t-1}P^{\beta_{n-1, j}}(-j)\oplus P^{\beta {n-1, n+t-1}}(-n-t+1)\oplus P^{n-1}(-n-s+1),\end{eqnarray*}
where $\beta_{n, n+t}\geq n.$ By Theorem \ref{asli2}  we get $\beta_n(R/I)>1$ which is a contradiction.
\end{proof}

\vskip 2mm
For instance, the above result says    that $
\{(1,3,4,4,1,1,1)\} $ cannot be the Hilbert function of any Artinian
Gorenstein local ring $A=K[[x,y,z]]/I$.
\\ In fact, $L=\Lex(I) $ should be  the ideal   $L=( x^2,xy,x^2z,xz^2,xyz,y^4,y^3z,y^2z^2,yz^3,z^7). $ The minimal free resolution of $P/L$ is:
\begin{eqnarray*}
& {\bf{G}}: 0& \to P(-5)\oplus P^3(-6)\oplus P(-9) \to P(-3)\oplus
P^2(-4)\oplus P^7(-5)\oplus P^2(-8) \to \\ && P^2(-2)\oplus
P(-3)\oplus P^4(-4)\oplus P(-7) \to P.
\end{eqnarray*}
However, if we  consider any sequence of  zero and  negative consecutive cancellations,
we get  $\beta_3(A)\geq 2$ which contradicts  the assumption that
$A$  is Gorenstein. \vskip 2mm
\bigskip

  We present now  an  investigation  in {\it codimension  two}. Assume $I$ is an ideal of a regular local ring $(R, \n) $ of dimension two such that $A=R/I$ is Artinian.

  By Macaulay's
Theorem, for some integer $d$ one has that the Hilbert function of $A$ is
 $$ \{(1,2,...,d,h_d,...,h_s,0,0,\dots)\}$$  where $s$ is the socle degree and
 $ d=h_{d-1}\geq h_d \geq h_{d+1} \geq ...\geq h_s \geq 1.$
We consider the corresponding lexicographic ideal $L$ in $P=k[x,y].$  It can be written as
$$L=( x^d,x^{d-1}y^{k_1},....,y^{k_d}) $$ with $0=k_0<k_1<...<k_d.$
For every $j>0$ define  $$e_j:=|h_j-h_{j-1}|.$$ We consider $h_j=j+1 $ if $j\le d $ and $h_j=0 $ if $j >s.$ It is easy to see
that the minimal number of generators of degree $d$ of $L$ is
$e_d+1$  and, for $j>d$, the minimal number of generators of degree
$j$ of $L$ is $e_j$. Notice that  $\sum_{j \ge d } e_j=d.$ The
minimal free resolution of $P/L$ is given by $\  0\to F_2 \to F_1
\to P \to P/L \to 0 $ where $ F_2 $ and $ F_1$  have  respectively
rank $d $ and $d+1. $ In particular,
\begin{equation} \label{F}  F_2= \oplus_{j\ge 0} P^{e_{d+j}}(-d-j-1)  \ \ \mbox{and} \ \ F_1= P^{e_{d}+1}(-d)   \oplus_{j\ge 1} P^{e_{d+j}}(-d-j ) .\end{equation}
 
As a consequence of Theorem \ref{asli2} we can present an easy proof  of  a result   stated by  Macaulay, later proved by Iarrobino in \cite{I} and by Bertella in \cite{B} with different methods and   technical devices.
\begin{corollary}
Let $ \{(1,2,...,d,h_d,...,h_s,0,0,\dots)\}$  be the Hilbert function of an Artinian  Gorenstein
  local ring $A=R/I$, then    $e_j \leq 1$ for
every $j>0. $
\end{corollary}
\begin{proof}
Let  $L$ be  the
  lexicographic ideal of $I$. With the above notations, we have
 $$ F_2= \oplus_{j\ge 0} P^{e_{d+j}}(-d-j-1)  \ \ \mbox{and} \ \ F_1= P^{e_{d}+1}(-d)   \oplus_{j\ge 1} P^{e_{d+j}}(-d-j )  $$
 with $e_{s+1}=1 $ ($h_s=1$) and $e_i=0 $ for $i\ge s+2.$
Since $R/I$  is Gorenstein,  $\beta_2(R/I)=1  $ and, by Theorem
\ref{asli2},  the   minimal free resolution of $R/I $ ($0 \to R \to
R^2 \to R $\ ) is obtained from the resolution of $L$ by a sequence
of zero and negative cancellations. Now $F_2= P(-s-2) \oplus
[\oplus_{d\le i \le s}  P^{e_i}(-i-1)].$ Since $P(-s-2) $  cannot be
cancelled,   we should cancel in $F_2$   the remaining  $
[\oplus_{d\le i \le s}  P^{e_i}(-i-1)]. $ Let $t \le  s  $ be  the
biggest integer such that $e_t \neq 0. $ The only chance  to cancel
$P^{e_t}(-(t+1) ) $ in $F_2$ is through $P(-s-1) $ in $F_1, $  hence
$e_t=1.$ Continuing this procedure $d-1$ times, we conclude that $
e_i \le 1 $ for every $i \ge 1.$

\end{proof}

    \vskip 2mm
 
 Following essentially the same idea as in \cite{B}, we can realize an ideal $I$ in $k[[x,y]] $ coming from any sequence of zero and negative cancellations in the resolution of $L=\Lex(I).  $

  \begin{remark} \label{due} {\rm  Let $0 \to F_2 \to F_1 \to L \to 0 $ be a minimal $P$-free resolution of $$L=( x^d,x^{d-1}y^{k_1},....,y^{k_d}) $$ with $0=k_0<k_1<...<k_d.$
 The matrix $M$ of the differential map $F_2 \to F_1$ is the well-known Hilbert-Burch matrix  whose  columns can be described as follows:
 \begin{equation} \label{HB} (\sigma_1^{tr},...,\sigma_d^{tr})\ \ \ {\text {where }}\ \ \sigma_i=(\stackrel{i-1}{\overbrace{{0,...,0}}},y^{k_i-k_{i-1}},-x,0,...,0)\ \ i=1,..,d. \end{equation}

  It  is clear that each cancellation in the resolution of $L$ corresponds   to an element in position
 $(i,j)$ in $M  $  with non-positive  value in the associated degree-matrix.  Consider  the sequence of cancellations corresponding to the elements in position $  \{(l_1,t_1),...,(l_k,t_k)\}   $ and let $I$ be the ideal of $R=k[[x,y]] $  generated by the maximal minors of the following  matrix
$$M' := M +(\alpha_{ij})_{\stackrel{ i=1,...,d+1}{j=1,...,d}}$$
where   $$\alpha_{ij}=\left\{%
\begin{array}{ll}
    1, & \hbox{if $(i,j)\in \{(l_1,t_1),...,(l_k,t_k)\}$;} \\
    0, & \hbox{otherwise.} \\
\end{array}%
\right.    $$ We can prove that $L=\Lex(I) $ and the corresponding
ideal $I^*$ follows by the above construction by performing  the
only  zero cancellations.  This construction holds in the general case of an Artinian local ring of codimension two, not necessarily Gorenstein,  by using the same procedure as in  \cite{B}, Theorem 2.4. }

\end{remark}

The next example realizes  the above construction.

\begin{example}  {\rm  Let $H=\{(1,2,3,4,3,3,3,2,2,1)\}, $ then $L=\langle
x^4,x^3y,x^2y^5,xy^8,y^{10}\rangle. $ The minimal free resolution of
$P/L$ is $$0 \to P(-5)\oplus P(-8)\oplus P(-10)\oplus P(-11) \to
P^2(-4)\oplus P(-7) \oplus P(-9)\oplus P(-10) \to P $$ The
resolution admits   two negative cancellations: $(P(-5),P(-7)),
(P(-8),P(-9))$ and the zero cancellation $(P(-10),P(-10)). $ Accordingly to Remark \ref{due},
the sequence obtained by the zero and negative cancellations     $0\to R\to
R^2\to R \    $   can be realized by the  ideal $I$ generated by the maximal minors of the following matrix
$$\left(%
\begin{array}{cccc}
  y & 0 & 0 & 0 \\
  -x & y^4 & 0 & 0 \\
  1 & -x & y^3 & 0 \\
  0 & 1 & -x & y^2 \\
  0 & 0 & 1 & -x \\
\end{array}%
\right).$$ We have  $I= \langle
x^4-x^2y^2-x^2y^3-x^2y^4+y^6,x^3y-xy^3-xy^4\rangle$ and
$L=\Lex(I).$   Furthermore,  the minimal free resolution of $P/I^*$ is obtained by performing the only zero cancellation:
 $$0 \to P(-5)\oplus P(-8)\oplus P(-11) \to
P^2(-4)\oplus P(-7) \oplus P(-9) \to P  $$ and  $I^*=
\langle x^4-x^2y^2,x^3y-xy^3,x^2y^5-y^7,xy^8\rangle$  is given by
the maximal minors of the matrix
$$\left(%
\begin{array}{cccc}
  y & 0 & 0 & 0 \\
  -x & y^4 & 0 & 0 \\
 0 & -x & y^3 & 0 \\
  0 & 0 & -x & y^2 \\
  0 & 0 & 1 & -x \\
\end{array}%
\right).$$
}

  \end{example}

 \end{document}